\documentclass{amsart}

\usepackage{amsmath, amssymb}
\usepackage{amscd}
\usepackage{verbatim}
\usepackage{tikz}
\usetikzlibrary{arrows,chains,matrix,positioning,scopes}

\makeatletter
\tikzset{join/.code=\tikzset{after node path={%
\ifx\tikzchainprevious\pgfutil@empty\else(\tikzchainprevious)%
edge[every join]#1(\tikzchaincurrent)\fi}}}
\makeatother

\tikzset{>=stealth',every on chain/.append style={join},
         every join/.style={->}}

\newtheorem{definition}{Definition}[section]
\newtheorem{theorem}[definition]{Theorem}
\newtheorem{lemma}[definition]{Lemma}
\newtheorem{corollary}[definition]{Corollary}
\newtheorem{proposition}[definition]{Proposition}
\theoremstyle{definition}
\newtheorem{remark}[definition]{Remark}

\newtheorem{example}[definition]{Example}

\newtheorem{notation}[definition]{Notation}

\newcommand{\clA}{\mathcal{A}}
\newcommand{\clB}{\mathcal{B}}
\newcommand{\clC}{\mathcal{C}}
\newcommand{\clH}{\mathcal{H}}
\newcommand{\clK}{\mathcal{K}}
\newcommand{\clI}{\mathcal{I}}
\newcommand{\clU}{\mathcal{U}}
\newcommand{\bh}{\mathcal{B}(\mathcal{H})}
\newcommand{\bk}{\mathcal{B}(\mathcal{K})}

\newcommand{\omax}{\otimes_{{\rm max}}}
\newcommand{\omin}{\otimes_{{\rm min}}}
\newcommand{\oc}{\otimes_{{\rm c}}}
\newcommand{\cstf}{\text{C}^*(\mathbb{F}_\infty)}

\newcommand{\cstu}{\text{C}^*_u}
\newcommand{\fn}{\mathbb{F}_n}

\newcommand{\A}{\mathcal{A}}
\newcommand{\B}{\mathcal{B}}
\newcommand{\C}{\mathcal{C}}
\newcommand{\K}{\mathcal{K}}
\newcommand{\osr}{\mathcal{R}}
\newcommand{\oss}{\mathcal{S}}
\newcommand{\ost}{\mathcal{T}}
\newcommand{\osw}{\mathcal{W}}
\newcommand\cstar{{\rm C}^*}
\newcommand\cstaru{{\rm C}_{\rm u}^*}
\newcommand\fni{{\mathbb F}_{\infty}}
\newcommand\coisubset{\subset_{\rm coi}}

\begin{document}

\title[Relative Weak Injectivity of Operator System Pairs]{Relative Weak Injectivity of Operator System Pairs}

\author[A.~Bhattacharya]{Angshuman Bhattacharya}

\address{Department of Mathematics and Statistics, University of Regina, Canada S4S 0A2}
\email{bhattaca@uregina.ca}

%\thanks{This work is part of the author's PhD thesis at the University of Regina and is supported in part
%by the Faculty of Graduate Studies and Research Dean's Scholarship and the Saskatchewan Innovation and Opportunity Scholarship}

\begin{abstract}
The concept of a relatively weakly injective pair of operator systems
is introduced and studied in this paper, motivated by
relative weak injectivity in the C*-algebra
category.  E. Kirchberg \cite{Kr} proved that the C$^*$-algebra $\cstar(\fni)$ of the
free group $\fni$ on countably many generators characterises relative weak injectivity
for pairs of C$^*$-algebras by means of the maximal tensor product. One of the main
results of this paper shows that $\cstar(\fni)$ also characterises relative weak injectivity
in the operator system category. A key tool is the theory of operator system tensor
products \cite{KP1,KP2}.
\end{abstract}

\maketitle

%%%%%%%%%%%%%

\section{Introduction}

A pair $(\A, \B)$ of unital C$^*$-algebras is a \emph{relatively weakly injective pair} for every unital C$^*$-algebra $\C$,
$\A\omax\C$ is a unital C$^*$-subalgebra of $\B\omax\C$. (In particular, one has that $\A$ is a unital C$^*$-subalgebra of $\B$.)
It is common to say that \emph{$\A$ is relatively weakly injective in $\B$} if the pair $(\A,\B)$ is a relatively weakly injective pair.
Relative weak injectivity for pairs of C$^*$-algebras was introduced by E.~Kirchberg \cite{Kr} and was motivated by the work of E.C.~Lance \cite{L} on the weak expectation property for C$^*$-algebras.

The purpose of this paper is to introduce and study a notion of relative weak injectivity for pairs $(\oss,\ost)$ of operator systems $\oss$ and $\ost$. To do so, one therefore needs to consider
operator system tensor products. Although the theory of tensor products \cite{KP1,KP2}  in the category $\mathcal O_1$, whose objects are operator systems and whose morphisms are unital completely positive (ucp) linear maps,
shares many similarities with C$^*$-algebraic tensor products, there some significant differences, particularly when considering the operator system analogue of the maximal C$^*$-algebraic tensor product, $\omax$.
With the max tensor product, there are two distinct tensor products (denoted by $\oc$ and $\omax$) in the category $\mathcal O_1$ that collapse to the maximal C$^*$-algebraic tensor product on the subcategory of
unital C$^*$-algebras and unital $*$-homomorphisms. In this paper an operator system analogue of relative weak injectivity will be developed using the commuting tensor product, $\oc$. Specifically,
a pair $(\oss,\ost)$ of operator systems is said to be a \emph{relatively weakly injective pair} if, for every operator system $\osr$,
$\oss\oc\osr$ is a unital operator subsystem of $\ost\oc\osr$.

The C$^*$-algebra $\cstar(\fni)$ of the free group $\fni$ on countably infinitely many generators is universal in the sense that every unital separable C$^*$-algebra is a quotient of $\cstar(\fni)$.
Therefore, it is striking that the C$^*$-algebra $\cstar(\fni)$ can be used to characterise both the weak expectation property and relative weak injectivity, as demonstrated by two important theorems of Kirchberg.
More precisely, $\A$ has WEP if and only if
$\A\omin\cstar(\fni)=\A\omax\cstar(\fni)$ \cite[Proposition 1.1]{Kr}, and
$(\A,\B)$ is a relatively weakly injective pair if and only if $\A\omax\cstar(\fni)\subset\B\omax\cstar(\fni)$ \cite[Proposition 3.1]{Kr}.

An operator system analogue of the weak expectation property for C$^*$-algebras--namely the double commutant expectation property--was introduced and
studied in \cite{K2,KP2}, and it was shown that $\cstar(\fni)$
characterises this property.
One of the main results of this paper shows that $\cstar(\fni)$ also characterises relative weak injectivity of operator system pairs
(Theorem \ref{rwi}). In addition to establishing some alternate characterisations of relative weak injectivity,
the existence of relatively weakly injective pairs $(\oss,\ost)$ in the operator system category will be achieved (in Theorem \ref{rwiex})
in a manner similar to Kirchberg's result \cite[Corollary 3.5]{Kr} that every unital separable C$^*$-algebra is a unital C$^*$-subalgebra of a unital separable C$^*$-algebra
with the weak expectation property.
The paper concludes with a selection of examples.

The theory of operator algebraic tensor products is treated in the books \cite{BO,Tak}, while operator system tensors products are
developed in the papers \cite{KP1,KP2}. Standard references for operator systems and completely positive maps are \cite{Pl,Pi}.

%%%%%%%%%%%%%%%%%%%%%%%
\section{The Commuting Operator System Tensor Product}

If $\oss$ and $\ost$ are operator systems, then the notation $\oss \subset \ost$ means
that $\oss$ is a unital operator subsystem of $\ost$. That is, if $1_\oss$ and $1_\ost$ denote the distinguished Archimedean order units for $\oss$ and $\ost$ respectively, then $1_\oss=1_\ost$.
Unless the context is not clear, the order unit for an operator system will be denoted simply by $1$.

The algebraic tensor product $\oss\otimes \ost$ of operator systems $\oss$ and $\ost$ is a $*$-vector space.
An \emph{operator system tensor product structure} on $\oss\otimes \ost$  is a family $\tau=\{\clC_n\}_{n\in \mathbb{N}}$ of cones $\clC_n\subset M_n(\oss\otimes\ost)$ such that:
\begin{enumerate}
\item $(\oss\otimes \ost, \tau, 1_\oss\otimes 1_\ost)$ is an operator system, denoted by $\oss\otimes_\tau\ost$, in which $1_\oss\otimes 1_\ost$ is an Archimedean order unit,
\item $M_n(\oss)_+ \otimes M_m(\ost)_+ \subset \clC_{nm}$, for all $n, m \in \mathbb{N}$, and
\item if $\phi : \oss \rightarrow M_n$ and $\psi : \ost \rightarrow M_m$ are unital completely positive (ucp) maps, then $\phi \otimes \psi : \oss\otimes_\tau \ost \rightarrow M_{nm}$ is a ucp map.
\end{enumerate}

Recall that a unital completely positive linear (ucp) map $\phi:\oss\rightarrow\ost$ of operator systems is a \emph{complete order isomorphism} if it is a linear bijection and if both $\phi$ and $\phi^{-1}$
are completely positive. If the ucp map $\phi$ is merely injective, then $\phi$ is a \emph{complete order injection} if $\phi$ is a complete order isomorphism of between $\oss$ and the operator subsystem $\phi(\oss)$ of $\ost$.

If $\oss_1\subset\ost_1$ and $\oss_2\subset\ost_2$ are inclusions of operator systems, and if  $\iota_j:\oss_j\rightarrow\ost_j$ are the
inclusion maps, then for any operator system structures $\tau$ and $\sigma$ on $\oss_1\otimes\oss_2$ and $\ost_1\otimes\ost_2$, respectively,
the notation (as used in \cite{FP3} also)
\[
\oss_1\otimes_\tau\oss_2\,\subset_+\,\ost_1\otimes_\sigma\ost_2
\]
expresses the fact that the linear vector-space embedding $\iota_1\otimes\iota_2:\oss_1\otimes\oss_2\rightarrow\ost_1\otimes\ost_2$
is a ucp map $\oss_1\otimes_\tau\oss_2\rightarrow\ost_1\otimes_\sigma\ost_2$.
That is,  $\oss_1\otimes_\tau\oss_2\,\subset_+\,\ost_1\otimes_\sigma\ost_2$ if and only if
$M_n(\oss_1\otimes_\tau\oss_2)_+\subset M_n(\ost_1\otimes_\sigma \ost_2)_+$ for every $n\in\mathbb N$.
If, in addition, $\iota_1\otimes\iota_2$ is a complete order isomorphism onto its range, then this is denoted by
\[
\oss_1\otimes_\tau\oss_2\coisubset\ost_1\otimes_\sigma\ost_2\,.
\]
Thus,
$\oss\otimes_\tau\ost\,=\,\oss\otimes_\sigma\ost$ means
$\oss\otimes_\tau\ost\coisubset\oss\otimes_\sigma\ost$ and
$\oss\otimes_\sigma\ost\coisubset\oss\otimes_\tau\ost$.

The commuting operator system tensor product $\oc$ was introduced and studied in \cite{KP1} and will be defined below. A slight simplification in the definition is afforded by the
following lemma, which allows one to restrict to ucp maps rather than use all completely positive maps.

\begin{lemma} \label{comm}
\cite[Lemma 2.2]{CE}, \cite[Lemma 5.1.6]{ER} Let $\oss \subset \bk$ be an operator system and $\phi : \oss \rightarrow \bh$ be a completely positive
map. Then there exists a ucp map $\tilde{\phi} :\oss \rightarrow \bh$ such that
$$\phi(\cdot) = \phi(1)^{\frac{1}{2}} \tilde{\phi}(\cdot) \phi(1)^{\frac{1}{2}}.$$
\end{lemma}

The proof of the lemma above describes the map $\tilde{\phi}$ as a strong limit of $\tilde{\phi}^{(n)}$ in $\bh$, where
$$\tilde{\phi}^{(n)}(s) = \left(\phi(1)+\frac{1}{n}\right)^{-\frac{1}{2}} \phi(s) \left(\phi(1)+\frac{1}{n}\right)^{-\frac{1}{2}} + \langle s \eta,\eta\rangle (1-\textrm{P}_{\phi(1)}),$$
for $\eta\in \clK$, and $\textrm{P}_{\phi(1)}$ is the projection onto the closure of the range of $\phi(1)$.
Thus, for operator systems $\oss\subset\clB(\clK_\oss)$ and $\ost \subset \clB(\clK_\ost)$,
if $\phi : \oss \rightarrow \bh$ and $\psi : \ost \rightarrow \bh$ are completely positive maps with commuting ranges, then the corresponding ucp maps $\tilde{\phi}$ and $\tilde{\psi}$ also have commuting ranges.

Denote by $\mbox{ucp}(\oss,\ost)$ the set of all pairs $(\phi, \psi)$ of ucp maps from $\oss$ and $\ost$, respectively, into $\clB(\clH)$ for some Hilbert space $\clH$, such that $\phi(\oss)$ commutes with $\psi(\ost)$. For each
$(\phi,\psi)\in \mbox{ucp}(\oss,\ost)$ let $\phi \cdot \psi : \oss\otimes \ost \rightarrow \clB(\clH)$ be the unique linear map whose value on elementary tensors is given by
$$\phi \cdot \psi( x \otimes y) = \phi(x)\psi(y).$$
Define cones by
$$\clC^{\textrm{comm}}_n =\{\eta \in M_n(\oss\otimes \ost) : (\phi \cdot \psi)^{(n)}(\eta) \geq 0, \textrm{ for all } (\phi,\psi)\in \mbox{ucp}(\oss,\ost)\}.$$
It was shown in \cite{KP1} that the collection of cones above is a matrix ordering on $\oss\otimes \ost$ with Archimedean matrix order unit $1_\oss\otimes 1_\ost$.

\begin{definition}
The operator system $(\oss\otimes \ost, \{\clC^{\textrm{comm}}_n\}_{n\in \mathbb{N}}, 1_\oss\otimes 1_\ost)$ is called the \emph{commuting operator system tensor product} of $\oss$ and $\ost$ and is denoted by $\oss\oc \ost$.
\end{definition}

The following notation, introduced in \cite{KP2}, will be used.
\begin{notation}\label{notation}
If $\mathcal X$ and $\mathcal Y$ are operator systems, then
$\mathcal X \hat{\otimes}_{\rm c}\mathcal Y$ shall denote the norm-completion of $\mathcal X \oc\mathcal Y$.  For any subspaces $\mathcal X_0\subset X$ and $\mathcal Y_0\subset \mathcal Y$,
$\mathcal X_0\overline{\otimes}\mathcal Y_0$ denotes the
closure of $\mathcal X_0{\otimes}\mathcal Y_0$ in $\mathcal X \hat{\otimes}_{\rm c}\mathcal Y$ .
\end{notation}

The symbol $\omax$ is reserved in this paper (unlike in \cite{KP1,KP2}) for the maximal C$^*$-algebra tensor product. An important fact:
if two unital C$^*$-algebras $\A$ and $\B$ are considered as operator systems, then
$\A\hat{\otimes}_{\rm c}\B=\A\omax\B$ \cite[Theorem 6.6]{KP1}.

In principle an abstract operator system $\oss$ generates many different C$^*$-algebras. The largest such C$^*$-algebra is called the universal C$^*$-algebra generated by $\oss$.
That is,
a unital C*-algebra $\clA$ is \emph{universal} for $\oss$ if:
\begin{enumerate}
\item there is a unital complete order injection $\iota_{\rm u}: \oss \rightarrow \clA$,
\item $\clA$ is generated by $\iota_{\rm u}(\oss)$, and
\item if $\phi : S \rightarrow \clB$ is a ucp map into another C*-algebra $\clB$, then there is a homomorphism $\pi : \clA \rightarrow \clB$ such that $\phi = \pi \circ \iota_{\rm u}$.
\end{enumerate}
It was shown in \cite[Proposition 8]{KrW} that every operator system has a universal C*-algebra, unique up to isomorphism, and an explicit construction was given. Therefore, $\cstaru(\oss)$ shall unambiguously denote
the universal C*-algebra generated by $\oss$.

\begin{theorem}\label{cinc} {\rm (\cite[Lemma 2.5]{KP2})} For all operator systems $\oss$ and $\ost$,
\[
\oss\oc\ost\coisubset\oss\oc\cstaru(\ost)\coisubset\cstaru(\oss)\omax\cstaru(\ost)\,.
\]
\end{theorem}

\begin{corollary} For every unital C$^*$-algebra $\A$, operator system $\oss$, and $n\in\mathbb N$, the operator systems $M_n(\oss\oc\A)$ and $\oss\oc M_n(\A)$ are completely order isomorphic.
\end{corollary}

%%%%%%%%%%%%%%%%%%%%%%%%%%%%%%
\section{Preliminary Results}

In this section we will use the fact that the matricial order on an operator system $\oss$ gives rise to a norm $\|\cdot\|_{M_n(\oss)}$ on each matrix space
$M_n(\oss)$ \cite[Chapter 3]{Pl}.

 \begin{lemma} \label{ocucp}
Let $\oss$ be an operator system and $\clA$ be a unital C*-algebra. A linear map $\phi : \oss\oc \clA \rightarrow \clB(\clH)$ is a ucp map if and only if there is a Hilbert space $\clK$,
homomorphisms $\pi : \cstaru(\oss) \rightarrow \clB(\clK)$ and $\rho : \clA \rightarrow \clB(\clK)$ with commuting ranges, and an isometry $V : \clH \rightarrow \clK$ such that $\phi (s\otimes a)=V^* \pi(s)\rho(a)V$ for all $s\in \oss$ and $a\in \clA$.
\end{lemma}

\begin{proof}
Because $\oss\oc\A\coisubset\cstaru(\oss)\omax\A$ by
Proposition \ref{cinc}, $\phi$ admits a ucp extension $\Phi:\cstaru(\oss)\omax\A\rightarrow \clB(\clH)$. Therefore, by \cite[Corollary 6.5]{KP1}, the restriction of $\Phi$ to $\oss\oc\A$
has the structure indicated in the statement of the lemma.
\end{proof}

\begin{lemma} \label{ultprod}
Let $\oss$ be a operator system. Let $\{\oss_i\}_{i\in \clI}$ be the set of all separable nontrivial operator subsystems of $\oss$ (that is, $\oss_i \subset \oss$). Then, there is a non-trivial ultrafilter $\clU$ on $\clI$ such that the map
$\Psi : \oss \rightarrow \prod^{\clU} \cstaru(\oss_i)$ given by $$x \longmapsto (\psi_i(x))_{\clU},$$ where $\psi_i(x)= x$ if $x\in \oss_i$ or $0$ otherwise, is a unital completely positive linear map, where $\prod^{\clU}$ denotes the C*-ultraproduct.
\end{lemma}

\begin{proof}
Note that the set $\clI$ is partially ordered by inclusion of the corresponding operator subsystems $\oss_i$ and that $\oss= \bigcup S_i$. Consider a cofinal ultrafilter $\clU$ on the directed set $\clI$.
The map $\Psi$ defined in the statement of the lemma is linear because of the structure of C$^*$-ultraproducts (see \cite{Had}).
To show that $\Psi$ is ucp it is sufficient to show that $\Psi$ is a complete isometry
(following the discussion after \cite[Remark 2.8.4]{Pi}).

If $x\in \oss$, note that the set $\{i\mid x\in \oss_i\} \in \clU$. To see this, simply observe that $\{i\mid x\in \oss_i\}=\{i\mid i\geq i_x\}$, where $\oss_{i_x}=\text{span}\{1,x,x^*\}$. Now, for $n=1$, $$\|\Psi(x)\|=\|(\psi_i(x))_{\clU}\|=\lim_{\clU}\|\psi_i(x)\|=\|x\|$$ by the preceding comment.

For $n>1$, we use a similar argument as follows. Let $X=(x_{kl}) \in M_n(\oss)$. Now, an ultrafilter is closed under finite intersections. So, $$I_X=\{i\mid x_{kl} \in \oss_i ~\forall~ k,l\}=\bigcap_{k,l}\{i\mid x_{kl}\in \oss_i\}$$ is in $\clU$. Finally, using the identification $M_n(\prod^{\clU} \cstaru(\oss_i))=\prod^{\clU} M_n(\cstaru(\oss_i))$ (see Remark on Pg-60 of \cite{Pi}) we obtain
\begin{eqnarray*}
\|\Psi^{(n)}(X)\|&=&\|(\Psi(x_{kl}))_{k,l}\|=\|((\psi_i(x_{kl}))_{\clU})_{k,l}\|
=\|((\psi_i(x_{kl}))_{k,l})_{\clU}\| \\ &=&\lim_{\clU}\|(\psi_i(x_{kl}))_{k,l}\|=\|(x_{kl})_{k,l}\|_{M_n(S_i), i\in I_X}=\|X\|\,,
\end{eqnarray*} thereby showing that $\Psi$ is a complete isometry.
\end{proof}

The following result is of central importance in what follows.

\begin{lemma} \label{lim}
Assume that $\clA$ is a C*-algebra and $\ost$ is an operator system, and fix $x\in \ost\otimes \clA$. If $\{\ost_i\}_{i\in \clI(x)}$
is the directed set of all separable unital operator subsystems of $\ost$ for which $x\in \ost_i\otimes \clA$, then
$$\|x\|_{\ost\oc \clA} = \lim_{\clI(x)} \|x\|_{\ost_i\oc \clA}.$$
\end{lemma}

\begin{proof}
Let us denote by $\|x\|_{(\cdot)}$ the norm $\|x\|_{(\cdot) \oc \clA}$. If $x\in\ost_1 \subset \ost_2$, then
\[
\ost_1\oc\A \subset_+ \ost_2\oc\A
\quad\mbox{implies that}\quad
\|x\|_{\ost_2} \leq \|x\|_{\ost_1}\,.
\]
Thus, $\lim_\clI \|x\|_{\ost_i}$ exists, since it is a decreasing net, and $$\|x\|_\ost \leq \lim_\clI \|x\|_{\ost_i}.$$
To establish the opposite inequality, following the techniques in the proof of \cite[Proposition 3.4]{Oz},we proceed as follows.

Assume that $\|x\|_{\ost_i} \geq 1$ for all $i\in \clI$. Thus, $\|x\|_{\ost_i} = \|x\|_{\cstu(\ost_i) \omax \clA} \geq 1$.
Therefore, there exists representations $\pi_i, \rho_i$ of $\cstaru(\ost_i)$ and $\clA$ respectively, on $\clB(\clH_i)$ with commuting ranges such that $$\| \pi_i \cdot \rho_i (x)\| \geq 1.$$
Using the map $\Psi$ from Lemma \ref{ultprod} above and the injective $*$-homomorphism $\iota : \clA \hookrightarrow \prod^{\clU} \clA$, where $\clU$ is the same ultrafilter over the same index set $\clI$ as in Lemma \ref{ultprod} or above, we have ucp maps
$\phi : \ost \rightarrow \clB(\clH_\ost)$ and $\rho : \clA \rightarrow \clB(\clH_\ost)$
with commuting ranges and such that $$\|\phi \cdot \rho (x)\| \geq 1,$$ where $\clH_\ost = \prod^{\clU} \clH_i$, $\phi = (\prod^{\clU} \pi_i)\circ \Psi$ and $\rho=(\prod^{\clU} \rho_i)\circ \iota$.

Now, $\phi \cdot \rho$ is a ucp map of $\ost \oc \clA$. By Lemma \ref{ocucp}, there exist representations $\pi_0$ and $\rho_0$ of $\cstaru(\ost)$ and $\clA$ with commuting ranges and an isometry $V$ such that $$\phi \cdot \rho (x) = V^* \pi_0 \cdot \rho_0 (x)V.$$
Since $\|\phi \cdot \rho (x)\| \geq 1$, we have $\|\pi_0 \cdot \rho_0 (x)\| \geq 1$ because $V$ is an isometry. But then,
$$\|x\|_\ost = \|x\|_{\cstaru(\ost) \omax \clA} \geq 1,$$ thereby showing that $\|x\|_\ost = \lim_\clI \|x\|_{\ost_i}$.
\end{proof}

\begin{remark}\label{limrmk}
Lemma \ref{lim} is also true if $\clA$ is only an operator system, as in that case, one may simply carry out the argument  above with $\cstaru(\clA)$ and arrive at the conclusion by virtue of Proposition \ref{cinc}.
\end{remark}

%%%%%%%%%%%%%%%%%%%%%%%%%%%%%%%%%%
\section{Main Results}

Recall that a
pair $(\oss,\ost)$ of operator systems is a relatively weakly injective pair if, for every operator system $\osr$,
\[
\oss\oc\osr\coisubset\ost\oc\osr\,.
\]
It is also convenient to say that
\emph{$\ost$ is relatively weakly injective in $\ost$} if $(\oss,\ost)$ is relatively weakly injective pair.

The first main result is an operator system version of Kirchberg's theorem \cite[Proposition 3.1]{Kr}.

\begin{theorem} \label{rwi} The following statements are equivalent for operator systems $\oss$ and $\ost$
for which $\oss\subset\ost$:
\begin{enumerate}
  \item $(\oss, \ost)$ is a relatively weakly injective pair of operator systems;
  \item $\oss \oc \cstar(\fni) \coisubset \ost \oc \cstar(\fni)$;
  \item For any ucp map $\phi :S \rightarrow \clB(\clH)$, there exist a ucp map $\Phi : \ost \rightarrow {\phi(\oss)}^{\prime\prime}$ such that $\Phi|_{\oss}=\phi$;
  \item $\left(\cstaru(\oss), \cstaru(\ost)\right)$ is a relatively weakly injective pair of C*-algebras.
  \end{enumerate}
\end{theorem}

\begin{proof} The order of implications to be proved is $(4) \Rightarrow (2) \Rightarrow (1) \Rightarrow (3) \Rightarrow (4)$.

$(4)\Rightarrow (2)$. Assume that $X\in M_n(\oss\otimes\cstar(\fni))$ is positive in $M_n(\ost\oc\cstar(\fni))$. We need to show that
$X\in M_n(\oss\oc\cstar(\fni))_+$. Because
\[
X\in M_n(\ost\oc \cstar(\fni))_+\subset M_n(\cstaru(\ost)\omax \cstar(\fni))_+\,,
\]
hypothesis (4) implies
 $X\in M_n(\cstaru(\oss)\omax \cstar(\fni))_+$, and so $X$ is positive in $M_n(\oss\oc \cstar(\fni))$ because
$\oss\oc\cstar(\fni)\coisubset\cstaru(\oss)\oc\cstar(\fni)$.

$(2)\Rightarrow (1)$.
Let $\osr$ be an arbitrary operator system. By Theorem \ref{cinc}, $\osw\oc\osr\coisubset\osw\oc\cstaru(\osr)$ for every operator system $\osw$;
thus, if we can show that $\oss\oc\cstaru(\osr)\coisubset\ost\oc\cstaru(\osr)$, then we deduce immediately that
$\oss\oc\osr\coisubset\ost\oc\osr$.

To begin, assume that $\osr$ is separable. Hence,
there is an ideal $\K$ of $\cstar(\fni)$ such that $\cstaru(\osr)=\cstar(\fni)/\K$.
By \cite[Corollary 5.17]{KP2}, and using Notation \ref{notation},
\[
\oss\oc\cstaru(\osr)\coisubset\oss\hat{\otimes}_{\rm c}\cstaru(\osr)\,=\, \frac{ \oss\hat{\otimes}_{\rm c}\cstar(\fni)}{\oss\overline{\otimes}\K}\,.
\]
The hypothesis $\oss \oc \cstf \coisubset \ost \oc \cstf$ implies that $\oss \oc \cstf \coisubset \cstaru(\ost) \oc \cstf$,
again by Theorem \ref{cinc}.
Therefore, \cite[Proposition 5.14]{KP2} yields
\[
\frac{ \oss\hat{\otimes}_{\rm c}\cstar(\fni)}{\oss\overline{\otimes}\K} \,\coisubset\, \frac{ \cstaru(\ost) \hat{\otimes}_{\rm c}\cstar(\fni)}{\cstaru(\ost) \overline{\otimes}\K}\,=\,\cstaru(\ost)\omax\cstaru(\osr)\,.
\]
Thus, $\oss\oc\cstaru(\osr) \coisubset \cstaru(\ost)\oc\cstaru(\osr)$, which implies $\oss\oc\cstaru(\osr)\coisubset\ost\oc\cstaru(\osr)$ and, hence,
$\oss\oc\osr\coisubset\ost\oc\osr$.

Now assume that $\osr$ is an arbitrary nonseparable operator system. We have proved above that $\oss\oc\osr_0\coisubset\ost\oc\osr_0$ for every separable operator system $\osr_0$.
Fix $x\in \oss\otimes \osr$ and choose a separable operator subsystem $\osr_1 \subset \osr$ such that $x\in \oss\otimes \osr_1$. Thus,
$\oss\oc \osr_1 \subset \ost\oc \osr_1$. By the beginning of the proof of Lemma \ref{lim} we have the inequality
$$\|x\|_{\oss\oc \osr} \leq \|x\|_{\oss\oc \osr_1} = \|x\|_{\ost\oc \osr_1}.$$
This inequality above holds  for any separable operator subsystem $\osr_1 \subset \osr$ for which $x\in \oss\otimes \osr_1$.
Lemma \ref{lim} (or Remark \ref{limrmk}) thus implies $\|x\|_{\oss\oc \osr} \leq \|x\|_{\ost\oc \osr}$, which in turn implies $$\|x\|_{\oss\oc \osr}=\|x\|_{\ost\oc \osr}.$$

Next, for $n>1$, fix $X \in M_n(\oss\otimes \osr)\subset M_n(\oss\otimes \cstaru(\osr)) \cong \oss\otimes M_n(\cstaru(\osr))$. One also
has $M_n(\oss\oc \cstu(\osr)) \cong \oss\oc M_n(\cstaru(\osr))$. Now, just as in the $n=1$ case, there exists a separable operator system $\osr_n^0 \subset M_n(\cstaru(\osr))$ such that $X\in \oss \otimes \osr_n^0$
and therefore, for any separable operator system $\osr_n \subset M_n(\cstaru(\osr))$ for which $X\in \oss \otimes \osr_n$, we have the inequality
$$\|X\|_{M_n(\oss\oc \cstaru(\osr))} = \|X\|_{\oss\oc M_n(\cstaru(\osr))} \leq \|X\|_{\oss \oc \osr_n}=\|X\|_{\ost \oc \osr_n}.$$
This implies (as in case of $n=1$) that
$$\|X\|_{M_n(\oss\oc \cstaru(\osr))} \leq \|X\|_{\ost\oc M_n(\cstaru(\osr))}= \|X\|_{M_n(\ost\oc \cstaru(\osr))},$$
which in turn implies that $\|X\|_{M_n(\oss\oc \cstaru(\osr))} = \|X\|_{M_n(\ost\oc \cstaru(\osr))}$. That is, the inclusion map $\oss\otimes\osr\rightarrow\ost\otimes\osr$
is a unital complete isometry $\oss\oc\osr\rightarrow\ost\oc\osr$ and, hence, is a complete order injection.

$(1)\Rightarrow (3)$. Let $\phi :\oss \rightarrow \clB(\clH)$ be a ucp map. Since $(\oss,\ost)$ is a relatively weakly injective pair, and because the commutant $\phi(\oss)^{\prime} \subset \clB(\clH)$ of $\phi(\oss)$ is a C$^*$-algebra,
$$\oss \oc \phi(\oss)^{\prime} \coisubset \ost \oc \phi(\oss)^{\prime} \coisubset \cstaru(\ost)\omax \phi(\oss)^{\prime}.$$
By the definition of commuting tensor product, $\phi \cdot \text{id}_{\phi(\oss)^\prime}$ is a ucp map on $\oss \oc \phi(\oss)^{\prime}$ with values in $\clB(\clH)$.
Take an Arveson extension $\Psi$ of $\phi \cdot \text{id}_{\phi(\oss)^\prime}$ to $\cstaru(\ost) \omax \phi(\oss)^\prime$ and
define a ucp map $\Phi$ on $\ost$ by $$\Phi(t) = \Psi(t \otimes 1),$$for all $t\in \ost$. Obviously, $\Phi|_\oss = \phi$. Finally, to see that $\Phi$ takes values in $\phi(\oss)^{\prime\prime}$, one
invokes the usual multiplicative domain argument for completely positive maps. This concludes our claim $(1)\Rightarrow (3)$.

$(3)\Rightarrow (4)$. Since $\oss\subset \ost$, $\cstaru(\oss)$ is a unital C*-subalgebra of $\cstaru(\ost)$ \cite[Proposition 9]{KrW}.
Let $\pi_U:\cstaru(\oss) \rightarrow \clB(\clH_U)$ be the universal representation of $\cstaru(\oss)$. Then $\pi_U|_{\oss}:\oss \rightarrow \clB(\clH_U)$ is a ucp map.
By hypothesis, $\pi_U|_{\oss}$ extends to $\phi : \ost \rightarrow (\pi_U|_{\oss}(\oss))^{\prime\prime} \subset (\pi_U(\cstaru(\oss)))^{\prime\prime}$. Now, since $\cstaru(\ost)$ is generated as an algebra by $\ost$,
the unique homomorphism from $\cstaru(\ost)$ extending $\phi$ takes values in $(\pi_U(\cstaru(\oss)))^{\prime\prime}$. Further, since this homomorphism extends $\pi_U|_\oss$, it fixes $\pi_U$, which completes the proof.
\end{proof}

The second main result shows the abundant existence of pairs of relatively weakly injective operator systems and is a generalisation of \cite[Lemma 3.4]{Kr}.

\begin{theorem} \label{rwiex} If $\oss$ is a separable operator subsystem of an operator system $\ost$, then
there exists a separable operator system $\osr$ such that $\oss \coisubset \osr \coisubset \ost$ and $\osr$ is relatively weakly injective in $\ost$.
\end{theorem}

\begin{proof}
Let $\{s_k\}_{k \in \mathbb{N}}$ be a dense sequence in $\oss \oc \cstar(\fni)$. Using Lemma \ref{lim}, we choose separable operator subsystems $\oss_n$
of $\ost$ such that, $\oss \subset \oss_1 \subset \oss_2 \subset \hdots $ and $\|s_k\|_{\oss_n} \leq \|s_k\|_{\ost} + \frac{1}{n}$ for $1 \leq k \leq n$. Let $\oss^{(1)}=\overline{\bigcup \oss_i}$.
Then $\oss^{(1)}$ is a separable operator system containing $\oss$, such that, for all $x \in \oss\oc \cstar(\fni)$, one has $\|x\|_{\oss^{(1)}}=\|x\|_\ost$.
By iterating the argument above with $\oss^{(1)}$ instead of $\oss$ we obtain a sequence of separable operator systems $\oss\subset \oss^{(1)} \subset \oss^{(2)}\subset \hdots$
such that $\|\cdot\|_{\oss^{(n)}}=\|\cdot\|_{\ost}$ on $\oss^{(n-1)} \oc \cstar(\fni)$. Define $\mathcal X_1 = \overline{\bigcup \oss^{(k)}}$. Thus, $\mathcal X_1$ is a separable operator system containing $\oss$
such that $\|\cdot\|_{\mathcal X_1}=\|\cdot\|_{\ost}$.

Replacing $\mathcal X_1$ for $\oss$ and $M_2(\cstar(\fni))$ for $\cstar(\fni)$, repeat the procedure described above to
obtain a separable operator system $\mathcal X_2$ such that, for all $x\in \mathcal X_2 \oc M_2(\cstar(\fni))$, we have
$$\|x\|_{\mathcal X_2 \oc M_2(\cstar(\fni))}=\|x\|_{\ost \oc M_2(\cstar(\fni))}.$$ In other words,
using the identification $\osw \oc M_2(\cstar(\fni))= M_2(\osw \oc\cstar(\fni))$ for operator systems $\osw$, we have that
the inclusion map $\mathcal X_2 \oc \cstar(\fni) \rightarrow \ost \oc \cstar(\fni)$ is a 2-isometry.

Further iterations of the procedure above gives us $\oss\subset \mathcal X_1 \subset \mathcal X_2\subset \mathcal X_3 \subset \hdots \ost$ such that the inclusion map
$\mathcal X_k \oc \cstf \rightarrow \ost \oc \cstf$ is a $k$-isometry.

Finally, set $\osr=\bigcup \mathcal X_k$. To show that $\osr$ is relatively weakly injective in $\ost$, it is enough, by Theorem \ref{rwi}, to show that the inclusion map $\osr \oc \cstf \rightarrow \ost \oc \cstf$ is a complete isometry.

For $Y\in\osr \otimes M_{n}(\cstf)$ there exists an integer $k_Y > n$ such that $Y\in \mathcal X_k \otimes M_{n}(\cstf)$ for all $k>k_Y$.
Now recall the fact that the inclusion maps $\mathcal X_k \oc \cstf \rightarrow \ost \oc \cstf$ are $k$-isometries. As a consequence, for $n<k_Y<k$ the inclusions
$\mathcal X_k \oc \cstf \rightarrow \ost \oc \cstf$ are also $n$-isometries.
Therefore, by Lemma \ref{lim} we have
\begin{eqnarray*}
\|Y\|_{\osr \oc M_{n}(\cstf)}&=& \lim_k \|Y\|_{\mathcal X_k \oc M_{n}(\cstf)}\\ &=&\lim_{k>k_Y} \|Y\|_{\mathcal X_k \oc M_{n}(\cstf)}\\
&=&\|Y\|_{\ost \oc M_{n}(\cstf)}.
\end{eqnarray*}
This shows that $\osr$ is relatively weakly injective in $\ost$, contains $\oss$, and is separable, thereby concluding the proof.
\end{proof}

\section{Remarks on relative weak injectivity with respect to the operator system maximum tensor product}

The maximal C$^*$-tensor product has two \textit{distinct} generalizations in the $\mathcal O_1$ category, namely the commuting tensor product and the operator system maximal tensor product. See \cite{KP1,KP2} for details. This article focuses on relative weak injectivity with respect to the former. A natural question would be to seek characterisations of relatively weakly injective operator system pairs with respect to the operator system maximal tensor product. Let us denote the operator system maximal tensor product by $\otimes_{\textrm m}$.

\begin{proposition} \label{rwiopmax}
Let $\oss\coisubset\ost$. The following statements are equivalent :
\begin{enumerate}
\item For any operator system $\osr$, $\oss\otimes_{\textrm m}\osr \coisubset \ost\otimes_{\textrm m}\osr$.
\item There exists a ucp map $\Phi : \ost \rightarrow \oss^{**}$, such that $\Phi(s) = s$ for all $s\in\oss$.
\end{enumerate}
\end{proposition}

\begin{proof}
$(1)\Rightarrow (2)$. Consider the bidual inclusion $\oss^{**}\coisubset\ost^{**}\coisubset\bh$, where the second inclusion is weak$^*$-WOT homeomorphic exactly as in the proof of \cite[Theorem 4.1]{Han}. Repeating the proof of \cite[Theorem 4.1 (iii)$\Rightarrow$(iv)]{Han} verbatim gives the required result.

$(2)\Rightarrow (1)$. For $X\in M_n(\ost\otimes_{\textrm m}\osr)^+ \cap M_n(\oss\otimes\osr)$, one has $X=(\Phi\otimes\textrm{id})^{(n)}(X)\in M_n(\oss\otimes_{\textrm m}\osr)\coisubset M_n(\oss^{**}\otimes_{\textrm m}\osr)$, where the last inclusion is due to \cite[Lemma 6.5]{KP2}.
\end{proof}

\begin{remark}
Comparing Proposition \ref{rwiopmax} and Theorem \ref{rwi}, it is unlikely that a universal characterisation of the likes of Theorem \ref{rwi}(2) exists in the $\otimes_{\textrm m}$ case. As a consequence, it cannot be ascertained that an existence result similar to Theorem \ref{rwiex} holds for the maximal operator system tensor product.
\end{remark}

\section{Examples}

\subsection{Operator systems generated by free unitaries}
Denote the generators of the free group $\fni$ by $\{u_j\}_{j\in\mathbb N}$. In $\cstar(\fni)$, each $u_j$ is a unitary and so, for each $n\in\mathbb N$, define
$$\mathbb{S}_n = \textrm{span}\{u_{-n},\hdots,u_{-1}, 1, u_1,\hdots, u_n\},$$
which is an operator subsystem of $\cstar(\fn)$.

\begin{example} \emph{For $n\in\mathbb{N}$, the pair $(\mathbb{S}_n,\cstar(\fn))$ is a relatively weakly injective pair of operator systems.}

The proof of this assertion is adapted from the proof of \cite[Lemma 4.1]{FP2} and makes use of our main result, Theorem \ref{rwi}.
 Let $\phi : \mathbb{S}_n \rightarrow \bh$ be a ucp map. By Theorem \ref{rwi}, it is enough to show that $\phi$ extends to $\cstar(\fn)$, taking values in $\phi (\mathbb{S}_n)^{\prime\prime}$.
 For each contraction $\phi(u_i)$, $1\leq i\leq n$, consider its Halmos unitary dilation $W_i$ on $\clH \oplus \clH$ given by
$$W_i = \left[
          \begin{array}{cc}
            \phi(u_i) & (1-\phi(u_i)\phi(u_{-i}))^{\frac{1}{2}} \\
            (1-\phi(u_{-i})\phi(u_i))^{\frac{1}{2}} & -\phi(u_{-i}) \\
          \end{array}
        \right]
$$
Let $T \in \phi (\mathbb{S}_n)^\prime$ and consider the operator $\tilde{T} = \left[
                                                                 \begin{array}{cc}
                                                                   T & 0 \\
                                                                   0 & T \\
                                                                 \end{array}
                                                               \right] \in \clB(\clH \oplus \clH)
$. Now, by functional calculus, $\tilde{T}$ commutes with $W_i$ for all $1\leq i\leq n$. Since $u_1,\dots,u_n$ are universal unitaries in $\cstar(\fn)$,
there is a unique homomorphism $\pi:\cstar(\fn)\rightarrow \clB(\clH \oplus \clH)$, such that $\pi(u_i) = W_i$ for $1\leq i\leq n$. Let $P=\left[
                                     \begin{array}{cc}
                                       I & 0 \\
                                       0 & 0 \\
                                     \end{array}
                                   \right]
$. Define ucp map $\tilde{\phi} : \cstar(\fn) \rightarrow \bh$ by $\tilde{\phi}(\cdot) = P\pi(\cdot)|_{\clH}$. Note that, $\tilde{\phi}$ extends $\phi$ and $\tilde{T}$ commutes with $P$.
Since, $\tilde{T}$ commutes with every $W_i$, it commutes with $\pi(\cstar(\fn))$. Thus, for $x\in \cstar(\fn)$ we have
$$\tilde{\phi}(x)T = P\pi(x)P \tilde{T}P = P\pi(x)\tilde{T}P= P\tilde{T}\pi(x)P = P\tilde{T} P\pi(x)P= T \tilde{\phi}(x).$$
So, $\tilde{\phi}(x) \in \phi (\mathbb{S}_n)^{\prime\prime}$ as $T$ was chosen arbitrarily in $\phi (\mathbb{S}_n)^\prime$. This concludes our claim.
\end{example}

\subsection{Operator systems generated from universal relations}
Let
$$\mathcal G=\{h_1,\dots,h_n\} \textrm{ and }
\mathcal R=\{h_j^*=h_j,\;\|h_j\|\leq 1,\;1\leq j\leq n\}$$
be a set of relations in the set $\mathcal G$, and let $\cstar(\mathcal G\vert\mathcal R)$ denote the universal unital C$^*$-algebra generated by $\mathcal G$ subject to $\mathcal R$.
The operator system
\[
NC(n)\,=\,\mbox{\rm span} \{ 1, h_1,...,h_n \}\,\subset\, \cstar(\mathcal G\vert\mathcal R)\,.
\]
is called the operator system of the non-commuting $n$-cube.

It was shown in \cite{FP3} that the C$^*$-envelope of $NC(n)$ is $\cstar(*_n\mathbb Z_2)$, where $*_n\mathbb Z_2$ is the free product of $n$-copies of $\mathbb Z_2$. The following example is from \cite[Lemma 6.2]{FP3} and can be proved exactly along the lines of the previous example.

\begin{example}\label{nc} \emph{For $n\in\mathbb{N}$, the pair $(NC(n),\cstar(*_n\mathbb Z_2))$ is a relatively weakly injective pair of operator systems.}
\end{example}

\subsection{Inclusion in the double dual}
The dual $\oss^{*}$ of an operator system is a matricially normed space, but the double dual $\oss^{**}$ is an operator system containing $\oss$ as an operator subsytem \cite{KP2}.
The following example is established in \cite[Corollary 6.6]{KP2}.

\begin{example}\label{dual} \emph{$(\oss,\oss^{**})$ is a relatively weakly injective pair of operator systems, for every operator system $\oss$.}
\end{example}

\subsection{Operator systems with DCEP}

An operator system $\oss$ is said to have the \emph{double commutant expectation property} (DCEP) if, for every complete order embedding $\oss\rightarrow \bh$, there exists a
completely positive linear map $\Phi : \bh \rightarrow\oss^{\prime\prime}\subset \bh$, fixing $\oss$.

\begin{example} \emph{ If $\oss$ has the double commutant expectation property, then $(\oss,\ost)$ is a relatively weakly injective pair of operator systems, for every operator system $\ost$ that contains $\oss$ as an operator subsystem.}

This assertion above is a consequence of \cite[Theorem 7.3, Theorem 7.1]{KP2}, which states that if $\oss\subset\ost$ and $\oss$ has the double commutant expectation property, then
$\oss\oc\osr\coisubset\ost\oc\osr$ for every operator system $\osr$.

\end{example}

\section{Acknowledgement}

The author would like to thank the referee for his/her valuable suggestions and his doctoral thesis advisor, Douglas Farenick, for suggesting the topic of this paper
and for many helpful discussions during the course of this work. The author's work at the University of Regina is supported in part by a Saskatchewan Innovation and Opportunity Scholarship
and a Faculty of Graduate Studies \& Research Dean's Scholarship.

%%%%%%%%%%%%%%%%%%%%%%%%%%%%%%%%%%%%%%%%%%%%%%%%%%%%%%%%%%%%%%%%%%%%

\end{document}